\documentclass[12pt]{amsart}
\usepackage{ifthen,verbatim}
\usepackage{floatflt,graphicx,graphics}
\usepackage{tikz}
\usepackage{mathrsfs}
\usepackage{amsmath,amsfonts,amssymb,amsthm}

\nonstopmode
\numberwithin{equation}{section}
\theoremstyle{plain}
\newtheorem{theorem}[equation]{Theorem}
\newtheorem{conjecture}[equation]{Conjecture}
\newtheorem{lemma}[equation]{Lemma}
\newtheorem{corollary}[equation]{Corollary}

\theoremstyle{definition}

\newtheorem{remark}[equation]{Remark}

\theoremstyle{remark}

\newcommand{\R}{\mathbb{R}}

\newcommand{\B}{\mathbb{B}}

\newcommand{\uhp}{\mathbb{H}}
\newcommand{\sh}{\mathrm{sh}\,}
\newcommand{\ch}{\mathrm{ch}\,}
\renewcommand{\th}{\mathrm{th}\,}

\def\Re{\mathop{\rm Re}\nolimits}
\def\Im{\mathop{\rm Im}\nolimits}

\DeclareMathOperator{\arsh}{arsh}

\newcounter{alphabet}

\newcounter{minutes}
\divide\time by 60
\newcounter{hours}
\multiply\time by 60
\addtocounter{minutes}{-\time}

\begin{document}
\bibliographystyle{amsplain}
\title
{Metrics and quasimetrics induced by point pair function
}

\def\thefootnote{}
\footnotetext{
\texttt{\tiny File:~\jobname .tex,
          printed: \number\year-\number\month-\number\day,
          \thehours.\ifnum\theminutes<10{0}\fi\theminutes}
}
\makeatletter\def\thefootnote{\@arabic\c@footnote}\makeatother

\author[D. Dautova]{Dina Dautova}
\author[S. Nasyrov]{Semen Nasyrov}
\author[O. Rainio]{Oona Rainio}
\author[M. Vuorinen]{Matti Vuorinen}

\begin{abstract}
We study the point pair function in subdomains $G$ of $\mathbb{R}^n$. We prove that, for every domain $G\subsetneq\mathbb{R}^n$, this function is a quasi-metric with the constant less than or equal to $\sqrt{5}\slash2$. Moreover, we show that it is a metric in the domain $G=\mathbb{R}^n\setminus\{0\}$ with $n\geq1$. We also consider generalized versions of the point pair function, depending on an arbitrary constant $\alpha>0$, and show that in some domains these generalizations are metrics if and only if $\alpha\leq12$.
\end{abstract}
\maketitle

\textbf{Author information.}

{\tiny
\noindent Dina Dautova$^{1}$ email: \texttt{dautovadn@gmail.com} ORCID: 0000-0002-7880-7598\\
Semen Nasyrov$^{1}$ email: \texttt{semen.nasyrov@yandex.ru} ORCID: 0000-0002-3399-0683\\
Oona Rainio$^{*2}$ email: \texttt{ormrai@utu.fi} ORCID: 0000-0002-7775-7656\\
Matti Vuorinen$^{2}$ email: \texttt{vuorinen@utu.fi} ORCID: 0000-0002-1734-8228\\
*: Corresponding author\\
1: Institute of Mathematics and Mechanics, Kazan Federal University, 420008 Kazan, Russia\\
2: Department of Mathematics and Statistics, University of Turku, FI-20014 Turku, Finland
}





\section{Introduction}


During the past few decades, several authors have contributed to the study of various metrics important for the
geometric function theory. In this field of research, \emph{intrinsic metrics} are the most useful because they measure distances in the way that takes into account not only how close the points are to each other but also how
the points are located with respect to the boundary of the domain. These metrics are often used to estimate the hyperbolic metric and, while they share some but not all of its properties, intrinsic metrics are much simpler than the hyperbolic metric and therefore more applicable.

Let $G$ be a proper subdomain of the real $n$-dimensional Euclidean space $\R^n$.  Denote by $|x-z|$ the Euclidean distance in $\mathbb{R}^n$ and by $d_G(x)$ the distance from a point $x\in G$ to the boundary $\partial G$, i.e. $d_G(x):=\inf\{|x-z|\text{ }|\text{ }z\in\partial G\}$. One of the most interesting  intrinsic measures of distance in $G$ is the \emph{point pair function} $p_G:G\times G\to[0,1)$ defined as
\begin{align}\label{ppf}
p_G(x,y)=\frac{|x-y|}{\sqrt{|x-y|^2+4d_G(x)d_G(y)}}\,,\quad x,y\in G.
\end{align}
This function was first introduced in \cite[p. 685]{chkv}, named in \cite{hvz} and further studied in \cite{hkvbook, fss, inm, sinb, sqm}. In \cite[Rmk 3.1, p. 689]{chkv} it was noted that the function $p_G$ is not a metric when the domain $G$ coincides with the unit disk $\B^2$.

In order to be a metric, a function needs to be fulfill certain three properties, the third of which is called the \emph{triangle inequality}. The point pair function has all the other properties of a metric but it only fulfills a relaxed version \eqref{ine_quasi} of the original triangle inequality, as explained in Section 2. We call such functions \emph{quasi-metrics} and study what is the best constant $c$ such that the generalized inequality \eqref{ine_quasi} holds. Namely, it was proven in \cite[Lemma 3.1, p. 2877]{fss} that the point pair function is a quasi-metric on every domain $G\subsetneq\R^n$ with a constant less than or equal to $\sqrt{2}$, but this result is not sharp for any domain $G$.

For this reason, we continue here the investigations initiated in the paper \cite{fss}. We give an answer to the question posed in  \cite[Conj. 3.2, p. 2877]{fss} by proving in Theorem \ref{thm_ppfquasiconstant} that, for all domains $G\subsetneq\R^n$, the point pair function is a quasi-metric with a constant less than or equal to $\sqrt{5}\slash2$. For the domain $G=\R^n\backslash\{0\}$ with $n\geq1$, we prove Theorem \ref{thm_pInR0}  which states that the point pair function $p_G$ defines a metric. In Lemma \ref{lem_csharpness}, we explain for which domains the constant $\sqrt{5}\slash2$ is sharp. 

We also investigate what happens when the constant $4$ in \eqref{ppf} is replaced by another constant $\alpha>0$ to define a generalized version $p^\alpha_G$ of the point pair function $p_G$ as in \eqref{ppfalpha}. In particular, we prove that, for $\alpha\in (0,12]$, this function $p^\alpha_G$ is a metric if 
$G$ is the positive real axis $\R^+$ (Theorem \ref{thm_alphappfR}), the punctured space $\R^n\backslash\{0\}$ with $n\geq2$ (Theorem \ref{pametr}), or the upper half-space $\uhp^n$ with $n\geq2$ (Theorem \ref{thm_mppfInH}). Furthermore, we also show in Theorem \ref{thm_alphappfnotmetricInB} that the function $p^\alpha_G$ is not a metric for any values of $\alpha>0$ in the unit ball $\B^n$.

The structure of this article is as follows. In Section 2, we give necessary definitions and notations. First,  in Section 3, we study the point pair function in the $1$-dimensional case and then consider the general $n$-dimensional case in Section 4. In Section 5, we inspect the generalized version $p^\alpha_G$ of the point pair function in several domains. At last, in Section 6, we state some open problems.


\section{Preliminaries}

In this section, we introduce some notation and recall a few necessary definitions related to metrics.

We will denote by $[x,y]$ the Euclidean line segment between two distinct points $x$, $y\in\R^n$. For every $x\in\R^n$ and $r>0$, $B^n(x,r)$ is the $x$-centered open ball of radius $r$, and
$S^{n-1}(x,r)$ is its boundary sphere. If $x=0$ and $r=1$ here, we simply write $\B^n$ instead of $B^n(0,1)$. Let $\uhp^n$ denote the upper-half space $\{x=(x_1,...,x_n)\in\R^n\text{ }|\text{ }x_n>0\}$. Furthermore, hyperbolic sine, cosine and tangent are denoted as sh, ch and th, respectively, and their inverse functions are arsh, arch, and arth.

For a non-empty set $G$, a \emph{metric} on $G$ is a function $d:G\times G\to[0,\infty)$ such that for all $x$, $y$, $z\in G$ the following three properties hold:\smallskip

\noindent (1) Positivity: $d(x,y)\geq 0$, and $d(x,y)=0$ if and only if $x=y$,

\noindent (2) Symmetry: $d(x,y)=d(y,x)$,

\noindent (3) Triangle inequality: $d(x,y)\leq d(x,z)+d(z,y)$.\smallskip

If a function $d:G\times G\to[0,\infty)$ 
satisfies (1),  (2) and the inequality 
\begin{align}\label{ine_quasi}
d(x,y)\leq c(d(x,z)+d(z,y))
\end{align}
for all $x$, $y$, $z\in G$ with some constant $c\ge1$ independent of the points $x,y,z$, then the function $d$ is a \emph{quasi-metric} \cite[p. 4307]{ps09}, \cite[p. 603]{s15}, \cite[Def. 2.1, p. 453]{x09}. Note that this term "quasi-metric" has slightly different meanings in some works, see for instance \cite{bm09,crt14,sra09}.

The point pair function defined in \eqref{ppf} is a metric for some domains $G\subsetneq\R^n$ and a quasi-metric for other domains \cite[Lemma 3.1, p. 2877]{fss}. Note that \emph{the triangular ratio metric} 
\begin{align*}
s_G(x,y)=\frac{|x-y|}{\inf_{z\in\partial G}(|x-z|+|z-y|)}\,,
\end{align*}
introduced by P. H\"ast\"o \cite{h}, is a metric for all domains $G\subsetneq\R^n$ \cite[Lemma 6.1, p. 53]{h}, \cite[p. 683]{chkv} and, because of the equality $p_{\uhp^n}(x,y)=s_{\uhp^n}(x,y)$ \cite[p. 460]{hkvbook}, the point pair function is a metric on $\uhp^n$. However, the point pair function is not a metric for either the unit ball \cite[Rmk 3.1, p. 689]{chkv} or a two-dimensional sector with central angle $\theta\in(0,\pi)$ \cite[p. 2877]{fss}.


\section{The point pair function in the one-dimensional case}


In this section, we prove that, for every $1$-dimensional domain $G$, the point pair function $p_G$ is either a metric or a quasi-metric with the sharp constant $\sqrt{5}\slash2$, depending on the number of the boundary points of $G$ (Corollary \ref{cor_1dimboundary}).

To do this, we need to establish the following lemma which is also required for the proof of another important result, Theorem~\ref{thm_quasi1disk}.

\begin{lemma}\label{lem_ineqppf}
Let the function $f:[-1,0]\times[0,1]\to\R$ be defined as
\begin{align*}
f(x,y)=\frac{y-x}{\sqrt{(y-x)^2+4(1+x)(1-y)}}
\quad\text{if}\quad
x\neq y,\quad f(x,y)=0\quad\text{otherwise},
\end{align*}
and, the function $g:[0,1]\times[0,1]\to\R$ be defined as
\begin{align*}
g(x,y)=\frac{y-x}{2-x-y}
\quad\text{if}\quad
x\neq y,\quad g(x,y)=0\quad\text{otherwise}.
\end{align*}
Then, for all $-1\le x\le0\le z\le y\le 1$, the inequality
\begin{equation}\label{ine0}
f(x,z)+g(z,y)\ge\frac{2}{\sqrt{5}}\,f(x,y)
\end{equation}
holds. Furthermore, the equality in \eqref{ine0} takes place if and only if $x=-1/3$, $z=0$ and $y=1/3$.
\end{lemma}

\begin{proof}
I) First we will investigate the function
$$
F(x,z,y)=f(x,z)+g(z,y)-\frac{2}{\sqrt{5}}f(x,y)
$$
in the domain $D=\{(x,z,y)\in\mathbb{R}^3\, :\, -1<x<0<z<y<1\}$.

By differentiation, we obtain
\begin{align*}
\frac{\partial f(x,y)}{\partial
x}&=-\,\frac{2(1-y)(2+x+y)}{(\sqrt{(y-x)^2+4(1+x)(1-y)})^3}\,,\\
\frac{\partial f(x,y)}{\partial
y}&=\frac{2(1+x)(2-x-y)}{(\sqrt{(y-x)^2+4(1+x)(1-y)})^3}\,,\\
\frac{\partial g(x,y)}{\partial x}&=-\,\frac{2(1-y)}{(2-x-y)^2}\,,\quad
\frac{\partial g(x,y)}{\partial y}=\frac{2(1-x)}{(2-x-y)^2}\,.
\end{align*}
Denote
$$
A=\sqrt{(z-x)^2+4(1+x)(1-z)}\,,
\quad B=\sqrt{(y-x)^2+4(1+x)(1-y)}\,.
$$
Then
\begin{align*}
\frac{\partial F(x,z,y)}{\partial x}&=\frac{\partial
f(x,z)}{\partial x}-\frac{2}{\sqrt{5}}\frac{\partial
f(x,y)}{\partial x}
=-\frac{2(1-z)(2+x+z)}{A^3}+\frac{2}{\sqrt{5}}\frac{2(1-y)(2+x+y)}{B^3}\,,\\
\frac{\partial F(x,z,y)}{\partial y}&=\frac{\partial
g(z,y)}{\partial y}-\frac{2}{\sqrt{5}}\frac{\partial
f(x,y)}{\partial
y}=\frac{2(1-z)}{(2-z-y)^2}-\frac{2}{\sqrt{5}}\frac{2(1+x)(2-x-y)}{B^3}\,,\\
\frac{\partial F(x,z,y)}{\partial z}&=\frac{\partial
f(x,z)}{\partial z}+\frac{\partial g(z,y)}{\partial
z}=\frac{2(1+x)(2-x-z)}{A^3}\,-\,\frac{2(1-y)}{(2-z-y)^2}\,.
\end{align*}
At every critical point of $F(x,y,z)$, we have $\nabla
F(x,y,z)=0$ and, therefore,
\begin{align}
-\frac{2(1-z)(2+x+z)}{A^3}\,+\,\frac{2}{\sqrt{5}}\frac{2(1-y)(2+x+y)}{B^3}&=0,\label{eq5}\\
\frac{2(1-z)}{(2-z-y)^2}\,-\,\frac{2}{\sqrt{5}}\,\frac{2(1+x)(2-x-y)}{B^3}
&=0,\nonumber\\
\frac{2(1+x)(2-x-z)}{A^3}-\frac{2(1-y)}{(2-z-y)^2}&=0.\nonumber
\end{align}
From the two latter equalities above, we can deduce that
\begin{align*}
A^3=\frac{(1+x)(2-x-z)(2-z-y)^2}{1-y}\,,\quad
B^3=\frac{2}{\sqrt{5}}\,\frac{(1+x)(2-x-y)(2-z-y)^2}{1-z}\,.
\end{align*}
By combining these expressions of $A$ and $B$ with the equality \eqref{eq5}, we have
\begin{align*}
\frac{A^3}{B^3}=\frac{\sqrt{5}}{2}\frac{(2-x-z)(1-z)}{(2-x-y)(1-y)}=\frac{\sqrt{5}}{2}\frac{(1-z)(2+x+z)}{(1-y)(2+x+y)}
\end{align*}
and, consequently, $x+y=x+z$. This implies that
$y=z$ and we see that $F(x,z,y)$ has no extrema in the
domain $D$.

\medskip

II) Now, let us investigate the case where $(x,z,y)$ is a boundary point of the aforementioned domain $D$. If $x=-1$, $x=0$, $z=y$ or $y=1$, then evidently $F(x,y,z)\ge0$. Thus, we only have to consider the case $z=0$. Without loss of generality, we can assume that $x>-1$ and $y<1$.

Since
\begin{align*}
f(x,0)=\frac{-x}{2+x},\quad g(0,y)=\frac{y}{2-y},
\end{align*}
the inequality \eqref{ine0} in the case $z=0$ is equivalent to the inequality
\begin{align}\label{ine6}
\frac{-x}{2+x}+\frac{y}{2-y}\ge\frac{2}{\sqrt{5}}\frac{y-x}{\sqrt{(y-x)^2+4(1+x)(1-y)}},
\quad -1<x<0<y<1.
\end{align}
By denoting $s=-x\slash(2+x)$ and $t=y\slash(2-y)$ for $0\le s,t\le 1$, we have
\begin{align*}
&x=\frac{-2s}{1+s},\quad
y=\frac{2t}{1+t},\quad
1+x=\frac{1-s}{1+s},\quad
1-y=\frac{1-t}{1+t},\quad
y-x=\frac{2(s+t+2st)}{(1+s)(1+t)},\\
&(y-x)^2+4(1+x)(1-y)=4\frac{(s+t+2st)^2+(1-s^2)(1-t^2)}{(1+s)^2(1+t)^2},
\end{align*}
and \eqref{ine0} is equivalent to
\begin{align*}
s+t\ge\frac{2}{\sqrt{5}}\,\frac{s+t+2st}{\sqrt{(s+t+2st)^2+(1-s^2)(1-t^2)}},\quad 0<s,t<1.
\end{align*}
This can also be written as
\begin{align}\label{ine7}
(s+t+2st)^2+(1-s^2)(1-t^2)\ge
\frac{4}{5}\,\left(\frac{s+t+2st}{s+t}\right)^2.
\end{align}
Now, let $u=s+t$ and $v=st$. Then $0<4v\le u^2<4$, and we can write \eqref{ine7} in the form
\begin{align*}
(u+2v)^2+1-u^2+2v+v^2\ge\frac{4}{5}\,\left(1+\frac{2v}{u}\right)^2.
\end{align*}
After simple transformations, we obtain
\begin{align*}
5v+4u+2+\frac{1}{5v}\,\ge\frac{16}{5}\,\left(\frac{1}{u}+\frac{v}{u^2}\right).
\end{align*}
Since $v\slash u^2\le1\slash4$, we only need to prove that
\begin{align*}
5v+4u+\frac{6}{5}+\frac{1}{5v}\,\ge\frac{16}{5}\frac{1}{u}
\quad\Leftrightarrow\quad
5uv+4u^2+\frac{6}{5}u+\frac{u}{5v}\,\ge\frac{16}{5}.
\end{align*}
Because $u\ge 2\sqrt{v}$, it is sufficient to establish
that
\begin{align*}
10v^{3/2}+16v+\frac{12}{5}v^{1/2}+\frac{2}{5v^{1/2}}\,\ge\frac{16}{5}.
\end{align*}
By denoting $v^{1/2}=\zeta$, we can write the last inequality in the form
\begin{align*}
h(\zeta)\equiv10\zeta^{4}+16\zeta^3+\frac{12}{5}\zeta^2-\frac{16}{5}\,\zeta+\frac{2}{5}\,\ge
0.
\end{align*}
It is easy to see that $h''(\zeta)\ge0$ for $\zeta\ge 0$, therefore $h$ is convex for positive $\zeta$. Consequently,
\begin{align*}
h(\zeta)=10(\zeta-1/5)^2\left((\zeta-1/5)^2+\frac{12}{5}(\zeta-1/5)+\frac{36}{25}\right).
\end{align*}
Since $h(1/5)=h'(1/5)=0$, the convexity of $h$ implies
$h(\zeta)\ge 0$ for $\zeta\ge 0$. Thus, the inequality \eqref{ine6} is proven. Furthermore, from the arguments above it follows that the equality in
\eqref{ine6} holds if and only if $\zeta=1/5$, and in this case $v=1/25$,
$u=2\sqrt{v}=2/5$, $s=t=1/5$ or, equivalently, $x=-1/3$ and $y=1\slash3$.
\end{proof}

\begin{theorem}\label{thm_quasi1disk}
The point pair function $p_G$ is a quasi-metric on the domain $G=(-1,1)\subset\R$ with a sharp constant $\sqrt{5}\slash2$.
\end{theorem}
\begin{proof}
We need to show that the function
\begin{equation}\label{alpha4}
p_G(x,y)=\frac{|x-y|}{\sqrt{|x-y|^2+4(1-|x|)(1-|y|)}}
\end{equation}
satisfies the inequality \eqref{ine_quasi} for all points $x,y,z\in(-1,1)$ with the constant $c=\sqrt{5}\slash2$. If $x$, $y$ and $z$ are all either non-negative or non-positive, then $p_G(x,y)\le p_G(x,z)+p_G(z,y)$ trivially. In the opposite case, either one of the points is negative and other two points are non-negative, or we have one positive and two non-positive points. Because of symmetry, we can just consider the
first possibility. If $z$ is negative, then the inequality
$p_G(x,y)\le p_G(x,z)+p_G(z,y)$ holds for all $x,y\in[0,1)$. Consequently, we can assume that $x<0\le z\le y$. In this case, our inequality can be simplified to
\begin{align*}
\frac{y-x}{\sqrt{(y-x)^2+4(1+x)(1-y)}}
\,\leq\,\frac{\sqrt{5}}{2}\,\left(
\frac{x-z}{\sqrt{(x-z)^2+4(1+x)(1-z)}}
+\frac{z-y}{2-z-y}\right).
\end{align*}
The inequality above follows from 
Lemma~\ref{lem_ineqppf}. Since,  in the case $x=-1\slash3$, $z=0$ and $y=1\slash3$,  the equality holds we see that the constant $\sqrt{5}\slash2$ is the best possible.
\end{proof}

We note that, for any $1$-dimensional domain $G\subsetneq\R$, the boundary $\partial G$ consists of either one or two points. Using this fact, we formulate:

\begin{corollary}\label{cor_1dimboundary}
If $G\subsetneq\R$ is a $1$-dimensional domain, then the point pair function $p_G$ is a metric if $\mbox{\rm card}(\partial G)=1$ and a quasi-metric if $\mbox{\rm card}(\partial G)=2$. Moreover, in the second case the best possible constant $c$ in the inequality $p_G(x,y)\le c(p_G(x,z)+p_G(z,y))$, $x,y,z\in G$,  equals $\sqrt{5}\slash2$.
\end{corollary}
\begin{proof}
First, we note that the point pair function is invariant under translation and stretching by a nonzero factor.

If card$(\partial G)=1$, then for some $x_0$ we have $\partial G=\{x_0\}$ and with the help of the function $f:x\mapsto a(x-x_0)$, $a\neq 0$,  we can map the domain $G$ onto the positive real axis $\R^+$. The function $f$ preserves the point pair function, i.e. $p_{\R^+}(f(x),f(y))=p_G(x,y)$ for all $x$, $y\in G$. Therefore,  from the very beginning we can assume that $G=\R^+$. Since $p_{\R^+}(x,y)=|x-y|\slash(x+y)$ coincides with the triangular ratio metric $s_{\R^+}(x,y)$ for all $x$, $y\in\R^+$, we conclude that in this case the point pair function is a metric.

If card$(\partial G)=2$, then we have $G=(x_0,x_1)$ for some $x_0$, $x_1\in\R$. Now, we apply the function $f:x\mapsto (2x-(x_0+x_1))/(x_1-x_0)$ which maps $G$ onto the interval $(-1,1)$. We see that, as above, $f$ preserves the point pair function, therefore we can assume that $G=(-1,1)$, and the result follows from Theorem~\ref{thm_quasi1disk}.
\end{proof}


\section{The point pair function in the $n$-dimensional case}


In this section, we investigate the quasi-metric property of the point pair function 
by considering its behaviour in $n$-dimensional domains, $n\geq1$. 
Our main results are
Theorems \ref{thm_pInR0} and \ref{thm_ppfquasiconstant}. First, we will establish 
Lemma~\ref{lem_trigABC}, which has quite complicated and technical inequalities but is necessary for the proof of Theorem \ref{thm_pInR0}. We note that some results close in spirit to those described in Lemmas~\ref{lem_trigABC} and \ref{style41} below are established in \cite{m}.

\begin{lemma}\label{lem_trigABC}
\textit{Let} $A=\sh (x+y)$,  $a\le\sh (u+v)$, $B=\sh x$, $C=\sh y$, $b=\sh u$, $c=\sh v$, and $x$, $y$, $u$, $v\ge 0$.  \textit{Then}
\begin{equation}\label{arsh}
\arsh\sqrt{A^2+a^2}
\le \arsh\sqrt{B^2+b^2}+\arsh\sqrt{C^2+c^2}
\end{equation}
\textit{or, what is equivalent, }
$$
\sqrt{A^2+a^2}
\le \sqrt{B^2+b^2}\sqrt{1+C^2+c^2}+ \sqrt{C^2+c^2}\sqrt{1+B^2+b^2}.
$$
Moreover,
\begin{equation}\label{aabbcc}
{\sqrt\frac{A^2+a^2}{1+A^2+a^2}}\,\le\, \sqrt {\frac{B^2+b^2}{1+B^2+b^2}}\,+\,\sqrt\frac{C^2+c^2}{1+C^2+c^2}\,.
\end{equation}

\end{lemma}

\begin{proof}
It is sufficient to prove that
\begin{multline}\label{shsqrt}
\sqrt{\sh^2(x+y)+\sh^2(u+v)}\le
\sqrt{\sh^2x+\sh^2u}\,\,\sqrt{1+\sh^2y
+\sh^2v}\,\,\\+\sqrt{\sh^2y+\sh^2v}\,\,\sqrt{1+\sh^2x+\sh^2u}.
\end{multline}
By squaring 
both sides of \eqref{shsqrt}, we have
\begin{multline*}
\sh^2(x+y)+\sh^2(u+v)\\ \le
(\sh^2x+\sh^2u)(1+\sh^2y+\sh^2v)+(\sh^2y+\sh^2v)(1+\sh^2x+\sh^2u)\\
+2\sqrt{\sh^2x+\sh^2u}\,\,\sqrt{1+\sh^2y+\sh^2v}\,\,\sqrt{\sh^2y+\sh^2v}\,\,\sqrt{1+\sh^2x+\sh^2u}
\end{multline*}
or
\begin{multline*}
(\sh x \,\ch y +\sh y \,\ch x)^2+(\sh u\,\ch v+\sh v\,\ch u)^2\\
\le\sh^2x(\ch^2y+\sh^2v)+\sh^2u(\ch^2v+\sh^2y)
+\sh^2y(\ch^2x+\sh^2u)+\sh^2 v(\ch^2u+\sh^2x)\\
+2\sqrt{\sh^2x+\sh^2u}\,\,\sqrt{1+\sh^2y+\sh^2v}\,\,\sqrt{\sh^2y+\sh^2v}\,\,\sqrt{1+\sh^2x+\sh^2u}.
\end{multline*}
After simple transformations, we obtain
\begin{multline*}
2\sh x \,\ch y\,\sh y \,\ch x+2\sh u
\,\ch v\,\sh v\,\ch u\\
\le\sh^2x\,\sh^2v+\sh^2u\,\sh^2y+\sh^2y\,\sh^2u+\sh^2 v+\sh^2x\\
+2\sqrt{\sh^2x+\sh^2u}\,\,\sqrt{1+\sh^2y+\sh^2v}\,\,\sqrt{\sh^2y+\sh^2v}\,\,\sqrt{1+\sh^2x+\sh^2u}.
\end{multline*}
To establish the inequality above, it is sufficient to prove that
\begin{multline*}
\sh x \,\ch y\,\sh y \,\ch x+\sh u
\,\ch v\,\sh v\,\ch u\\
\le\sqrt{\sh^2x+\sh^2u}\,\,\sqrt{1+\sh^2y+\sh^2v}\,\,\sqrt{\sh^2y+\sh^2v}\,\,\sqrt{1+\sh^2x+\sh^2u}.
\end{multline*}
Squaring this inequality, we have
\begin{align*}
&\sh^2 x  \,\ch^2 y \,\sh^2 y  \,\ch^2 x+\sh^2 u
 \,\ch^2 v \,\sh^2 v \,\ch^2 u+2  \,\sh x \,\ch y \,\sh y  \,\ch x \,\sh u
 \,\ch v \,\sh v \,\ch u\\
&\le(\sh^2x+\sh^2u)(1+\sh^2y+\sh^2v)(\sh^2y+\sh^2v)(1+\sh^2x+\sh^2u)
\end{align*}
or
\begin{multline}\label{mml}
\sh^2 x\,\ch^2 y\,\sh^2 y \,\ch^2 x+\sh^2 u
\,\ch^2 v\,\sh^2 v\,\ch^2 u+2 \,\sh x \,\ch y\,\sh y \,\ch x\,\sh u
\,\ch v\,\sh v\,\ch u
\\
\le
[\sh^2x(1+\sh^2y+\sh^2v)+\sh^2u(1+\sh^2y+\sh^2v)]\\ \times[\sh^2y(1+\sh^2x+\sh^2u)+\sh^2v(1+\sh^2x+\sh^2u)]
\\
=[\sh^2x\,\ch^2y+\sh^2x\,\sh^2v+\sh^2u\,\ch^2v+\sh^2u\,\sh^2y]\\ \times[\sh^2y\,\ch^2x+\sh^2y\,\sh^2u+\sh^2v\,\ch^2u+\sh^2v\,\sh^2x].
\end{multline}
By the inequality of arithmetic and geometric means, we have
$$
2 \sh x  \,\ch y \,\sh y  \,\ch x \,\sh u
\ch v \,\sh v \,\ch u\le \sh^2x \,\ch^2y \,\sh^2v \,\ch^2u+\sh^2u \,\ch^2v \,\sh^2y \,\ch^2x,
$$
therefore,
\begin{multline*}
\sh^2 x  \,\ch^2 y \,\sh^2 y  \,\ch^2 x+\sh^2 u
 \,\ch^2 v \,\sh^2 v \,\ch^2 u+2 \sh x  \,\ch y \,\sh y  \,\ch x \,\sh u
 \,\ch v \,\sh v \,\ch u\\
\le
[\sh^2x\ch^2y+\sh^2u\ch^2v][\sh^2y\ch^2x+\sh^2v\ch^2u].
\end{multline*}
This inequality implies \eqref{mml}, therefore, \eqref{arsh} is proved. The inequality \eqref{aabbcc} can be obtained by applying the function $\th$ to both 
sides of \eqref{arsh}.
\end{proof}

\begin{theorem}\label{thm_pInR0}
For $n\geq1$, the point pair function is a metric on $G=\R^n\backslash\{0\}$.
\end{theorem}
\begin{proof}
Because the point pair function trivially satisfies the properties (1) and (2) of a metric, we  only need to prove the triangle inequality. Therefore, we will show that $p_G(x,y)\leq p_G(x,z)+p_G(z,y)$ for $x$, $y$, $z\in G=\R^n\backslash\{0\}$. Note that, for all points $x$, $y$ in this domain,
\begin{align*}
p_G(x,y)\,=\,\frac{|x-y|}{\sqrt{|x-y|^2+4|x|\,|y|}}\,.
\end{align*}


1) First we consider the case $n=2$. Then we can identify points of $\R^2$ with complex numbers.

Because of homogeneity of $p_G(x,y)$, we can assume that $z=1$, so that the triangle inequality becomes
\begin{equation}\label{Rr1a}
\frac{|x-y|}{\sqrt{|x-y|^2+4|x|\,|y|}}\,\le\,
\frac{|x-1|}{\sqrt{|x-1|^2+4|x|}}\,+\,\frac{|1-y|}{\sqrt{|1-y|^2+4|y|}}\,.
\end{equation}
Let $x=Re^{i2\phi}$, $y=re^{i2\psi}$, $R$, $r>0$, $\phi$, $\psi\in \R$. We can assume that $R\ge r$.

First we will show that if either $0<r\le R\le 1 $ or $1\le r\le R$, then \eqref{Rr1a} holds.
Let us fix some $u$ and $v$ such that $x=u^2$, $y=v^2$. Then  \eqref{Rr1a} is equivalent to the inequality
\begin{equation}\label{Rr2a}
\frac{|u/v-v/u|}{\sqrt{|u/v-v/u|^2+4}}\,\le\,
\frac{|u-1/u|}{\sqrt{|u-1/u|^2+4}}\,+\,\frac{|v-1/v|}{\sqrt{|v-1/v|^2+4}}\,.
\end{equation}

If $0<r\le R\le 1 $, then $|u|$, $|v|\le1$ and
\begin{multline}\label{uv1}
|u/v-v/u|\le |u/v-uv|+|uv-v/u|\\ =|u|\,|v-1/v|+|v|\,|u-1/u| \le
|u-1/u|+|v-1/v|.
\end{multline}
Therefore, if we put
$$
p=\arsh\frac{|u/v-v/u|}{2}\,, \quad q=\arsh \frac{|u-1/u|}{2}\,, \quad  s=\arsh \frac{|v-1/v|}{2}\,,
$$
then, by \eqref{uv1}, we have $\sh p\le\sh q+\sh s$. But this immediately implies $p\le q+ s$ and $\th p\le \th q+ \th s$ what is equivalent to  \eqref{Rr1a}.

Since the inequality  \eqref{Rr2a} does not change after replacing $u$ and  $v$ with
$u^{-1}$ and $v^{-1}$, we see that for the case $1\le r\le R$ the inequality \eqref{uv1} is also valid.

Thus, we only need to consider the case $r\le 1\le R$.
We have
\begin{align*}
|x-y|&=\sqrt{R^2+r^2-2Rr\cos[2(\phi-\psi)]},\quad
|x-1|=\sqrt{R^2+1-2R\cos2\phi},\\
|1-y|&=\sqrt{r^2+1-2r\cos2\psi},
\end{align*}
consequently, the inequality \eqref{Rr1a} can be written in the form
\begin{multline*}
\sqrt{\frac{R^2+r^2-2Rr\cos[2(\phi-\psi)]}{R^2+r^2+2Rr(2-\cos[2(\phi-\psi)])}}\, \\ \le\,
\sqrt{\frac{R^2+1-2R\cos2\phi}{R^2+1+2R(2-\cos2\phi)}}\,+\,
\sqrt{\frac{r^2+1-2Rr\cos2\psi}{r^2+1+2r(2-\cos2\psi)}}\,.
\end{multline*}
This can be simplified to
\begin{multline}\label{Rr1}
\sqrt{\frac{(R-r)^2+4Rr\sin^2(\phi-\psi)}{(R+r)^2+4Rr\sin^2(\phi-\psi)}} \\
\,\le\,\sqrt{\frac{(R-1)^2+4R\sin^2\phi}{(R+1)^2+4R\sin^2 \phi}}\,+\,
\sqrt{\frac{(1-r)^2+4r\sin^2\psi}{(r+1)^2+4r\sin^2(\psi)}}\,.
\end{multline}

If we denote
\begin{equation}\label{ABCdef}
A=\frac{|R-r|}{2\sqrt{Rr}}\,, \quad B=\frac{|R-1|}{2\sqrt{R}}\,, \quad
C=\frac{|1-r|}{2\sqrt{r}}\,,
\end{equation}
then the inequality \eqref{Rr1} takes the form
\begin{equation}\label{ABC}
\sqrt{\frac{A^2+\sin^2(\phi-\psi)}{1+A^2+\sin^2(\phi-\psi)}}\,\le\,
\sqrt{\frac{B^2+\sin^2\phi}{1+B^2+\sin^2\phi}}\,+\,\sqrt{\frac{C^2+\sin^2\psi}{1+C^2+\sin^2\psi}}\,.
\end{equation}
Let $a=|\sin(\phi-\psi)|$, $b=|\sin\phi|$, $c=|\sin\psi|$, $u=\arsh b$ and $v=\arsh c$. Then $a\le \sh(u+v)$, since
\begin{align*}
a&=|\sin\phi\cos\psi-\sin\psi\cos\phi|
\le|\sin\phi\cos\psi|+|\sin\psi\cos\phi|
=b\sqrt{1-c^2}+c\sqrt{1-b^2}\\
&\le b\sqrt{1+c^2}+c\sqrt{1+b^2}.
\end{align*}

If $A$, $B$ and $C$ are as in \eqref{ABCdef}, then we have  $A=\sh(\arsh B+\arsh C)$. By applying the function $\th$ to the inequality \eqref{arsh} of Lemma \ref{lem_trigABC} and combining this with  the inequality
$\th(u+v)\le \th u+\th v$ for $u,v\ge 0$, we obtain~\eqref{ABC}.
\medskip

2) Now we consider the case $n\neq 2$. If $n=1$, then the statement of the theorem immediately follows from the case 1). Therefore, we will assume that $n\ge 3$. Consider the subspace $E$ of $\R^n$ containing the points $0$, $x$, $y$ and $z$. Since the Euclidean distance and the function $p_G$ are invariant under orthogonal transformations of $\R^n$ and for $n\ge 2$ the triangle inequality is valid, we can assume that $E$ coincides with $\R^3$.

Without loss of generality we can put $|z|=1$.  Now consider the vectors $Ox$, $Oy$, and $Oz$ from the origin to the points $x$, $y$, and $z$, respectively. Let $2\alpha$ be the angle between  $Ox$ and $Oz$, $2\beta$ be the angle between $Oz$ and $Oy$, and $2\gamma$  be the angle between  $Ox$ and $Oy$; $\alpha$, $\beta$, $\gamma\in[0,\pi/2)$. Then, by  the law of cosines,
\begin{align*}
|x-y|&=\sqrt{R^2+r^2-2Rr\cos2\gamma},\quad
|x-z|=\sqrt{R^2+1-2R\cos2\alpha},\\
|z-y|&=\sqrt{r^2+1-2r\cos2\beta},
\end{align*}
where $R=|x|$ and $r=|y|$.
Applying the same arguments as above in the case $n=2$, we see that we only  need to prove the inequality
\begin{equation}\label{ABC1}
\sqrt{\frac{A^2+\sin^2\gamma}{1+A^2+\sin^2\gamma}}\,\le\,
\sqrt{\frac{B^2+\sin^2\alpha}{1+B^2+\sin^2\alpha}}\,+\,\sqrt{\frac{C^2+\sin^2\beta}{1+C^2+\sin^2\beta}}\,,
\end{equation}
where $A$, $B$ and $C$  is defined by  \eqref{ABCdef}.

Denote $a=\sin\gamma$, $b=\sin\alpha$, $c=\sin\beta$. Consider the triangular angle, formed by the vectors $Ox$, $Oy$ and $Oz$. It has plane angles equal  $2\alpha$, $2\beta$ and $2\gamma$. Since each plane angle of a triangular angle is less than the sum of its other two plane angles, we obtain $2\gamma\le 2\alpha+2\beta$, therefore, $\gamma\le \alpha+\beta$.

Now, we will show that this implies the inequality $\sin\gamma\le \sin \alpha+\sin\beta$. Actually, if $\alpha+\beta\le \frac{\pi}{2}$, then $\sin \gamma\le \sin(\alpha+\beta)\le \sin \alpha+\sin\beta$. If $\alpha+\beta> \frac{\pi}{2}$, then $\beta> \frac{\pi}{2}-\alpha$ and $\sin \alpha+\sin\beta\ge \sin \alpha+\sin(\frac{\pi}{2}-\alpha)=\sin \alpha+\cos \alpha=\sqrt{2}\sin(\alpha+\frac{\pi}{4})\ge1\ge \sin \gamma$, since $\frac{\pi}{4}\le\alpha+\frac{\pi}{4}\le \frac{3\pi}{4}$.

Thus, we have $a\le b+c$ and this implies that $a\le b\sqrt{1+c^2}+c\sqrt{1+b^2}$, and we can continue the proof just as in the Case~1) to show that \eqref{ABC1} is valid.
\end{proof}

\begin{theorem}\label{thm_ppfquasiconstant}
On every domain $G\subsetneq\R^n$, $n\geq1$, the point pair function $p_G$ is a quasi-metric with a constant less than or equal to $\sqrt{5}\slash2$.
\end{theorem}
\begin{proof}

To prove that the point pair function  $p_G$ is a quasi-metric, we  only need to find such a constant $c\geq1$ that 
$$
p_G(x,y)\leq c(p_G(x,z)+p_G(z,y))
$$
for all points $x$, $y$, $z\in G$.
Let
$$
c(x,y,z;G)=\frac{p_G(x,y)}{p_G(x,z)+p_G(z,y)}\,,
$$
where $x$, $y$, and $z$ are distinct points from $G$. 
Define
\begin{equation}\label{c*}
c^*=\sup_{x,y,z,G}c(x,y,z;G),
\end{equation}
where the supremum is taken over all domains $G\subsetneq\R^n$ and triples of distinct points. We will call such domains and triples admissible.
We need to prove that $c^*=\sqrt{5}/2$.

Let us fix a domain $G\subsetneq\R^n$ and two distinct points $x$, $y\in G$. Since $G\neq \R^n$, the boundary $\partial G\neq\emptyset$, therefore, there exist points $u$, $v\in \partial G$ such that $d_G(x)=|x-u|$ and $d_G(y)=|y-v|$. In the general case, the points $u$ and $v$ might not be unique because there can be several boundary points on the spheres $S^{n-1}(x,d_G(x))$ and $S^{n-1}(y,d_G(y))$.


We note that the value $p_G$ decreases as $G$ grows, i.e. if $G\subset G_1$, then $p_{G_1}(\widetilde{x},\widetilde{y})\le p_G(\widetilde{x},\widetilde{y})$ for all $\widetilde{x}$, $\widetilde{y}\in G$.

Consider the two following cases.  \smallskip

1) If $u=v$, then we can set $G_1=\R^n\backslash \{u\}$. It is clear that $G\subset G_1$ and  $p_G(x,y)=p_{G_1}(x,y)$. Taking into account the invariance of $p_G$ under the shifts of $\R^n$, we can assume that $u=0$. Then, by Theorem~\ref{thm_pInR0}, we have $p_{G_1}(x,y)\le p_{G_1}(x,z)+p_{G_1}(z,y)$ for all $z\in G_1$. From the monotonicity of $p_G$ with respect to $G$, we obtain
$$
p_{G}(x,y)=p_{G_1}(x,y)\le p_{G_1}(x,z)+p_{G_1}(z,y)\le p_{G}(x,z)+p_{G}(z,y).
$$
Therefore, $c(x,y,z;G)\le 1$.


\smallskip

2) Let now $u\neq v$. We put $G_1=G_1^{u,v}=\R^n\backslash \{u,v\}$.  Then $G\subset G_1$. Moreover, $p_{G}(x,y)=p_{G_1}(x,y)$ and $p_{G_1}(x,z)+p_{G_1}(z,y)\le p_{G}(x,z)+p_{G}(z,y)$. Consequently, $c(x,y,z;G)\le c(x,y,z;G_1)$ and the supremum in \eqref{c*} is attained on domains of the type $G_1^{u,v}$.

Denote $a=|x-z|$, $b=|z-y|$, $c=|x-y|$, $\rho=|x-u|$, $r=|y-v|$. By the triangle inequality, we have
\begin{align*}
c\le a+b,\quad \rho&=|x-u|\le |x-v|\le c+r\le a+b+r, \\ r&=|y-v|\le |y-u|\le c+\rho\le a+b+\rho.
\end{align*}
Now consider the segment $\Delta$ on $\R^1$ with endpoints $\widetilde{u}:=-a-\rho$, $\widetilde{v}:=b+r$. Let $\widetilde{z}=0$, $\widetilde{x}=-a$, $\widetilde{y}=b$. Then $\widetilde{u}<\widetilde{x}<\widetilde{z}<\widetilde{y}<\widetilde{v}$ and
$$
|\widetilde{x}-\widetilde{u}|=\rho=|x-u|,\quad |\widetilde{y}-\widetilde{v}|=r=|y-v|,\quad |\widetilde{x}-\widetilde{y}|=a+b\ge c=|x-y|,
$$
$$
|\widetilde{x}-\widetilde{z}|=a=|x-z|,\quad |\widetilde{z}-\widetilde{y}|=b=|z-y|.
$$
With the help of the triangle inequality, we also have
$$
|\widetilde{x}-\widetilde{v}|=a+b+r\ge |x-u|=|\widetilde{x}-\widetilde{u}|, \quad |\widetilde{y}-\widetilde{u}|=a+b+\rho\ge |y-v|=|\widetilde{y}-\widetilde{v}|.
$$
Therefore, $d_\Delta(\widetilde{x})=\rho$, $d_\Delta(\widetilde{y})=r$.

At last,
$$
|\widetilde{z}-\widetilde{u}|=a+\rho\ge |z-u|,\quad |\widetilde{z}-\widetilde{v}|=b+r\ge |z-v|,
$$
and this implies
$d_\Delta(\widetilde{z})\ge d_{G_1}(z)$.
Using the obtained inequalities and the fact that the function $t\mapsto t/\sqrt{t^2+\gamma^2}$ is increasing on $\R_+$ when $\gamma$ is a real nonzero constant, we have
$$
p_\Delta(\widetilde{x},\widetilde{y})=\frac{|\widetilde{x}-\widetilde{y}|}{\sqrt{|\widetilde{x}-\widetilde{y}|^2+4 \rho r}}\ge\frac{|{x}-{y}|}{\sqrt{|{x}-{y}|^2+4 \rho r}}=
p_{G_1}({x},{y}),$$
$$
p_\Delta(\widetilde{x},\widetilde{z})=\frac{|\widetilde{x}-\widetilde{z}|}{\sqrt{|\widetilde{x}-\widetilde{z}|^2+4 \rho d_\Delta(\widetilde{z}) }}=
\frac{|{x}-{z}|}{\sqrt{|{x}-{z}|^2+4 \rho d_\Delta(\widetilde{z}) }}
$$$$\le\frac{|{x}-{z}|}{\sqrt{|{x}-{z}|^2+4 \rho d_{G_1}(z)}}=
p_{G_1}({x},{z})
$$
and,
similarly, $p_\Delta(\widetilde{z},\widetilde{y})\le p_{G_1}({z},{y})$. From this, we deduce that $c(x,y,z;G_1)\le c(\widetilde{x},\widetilde{y},\widetilde{z};\Delta)$.

Since the point pair function is invariant under shifts and stretchings, we can assume that $\Delta=[-1,1]$. But, by Theorem \ref{thm_quasi1disk}, the point pair function fulfills the inequality
\begin{equation}\label{cp}
p_\Delta(\widetilde{x},\widetilde{y})\leq \frac{\sqrt{5}}{2}\,(p_\Delta(\widetilde{x},\widetilde{z})+p_\Delta(\widetilde{z},\widetilde{y}))
\end{equation}
for all points $\widetilde{x}$, $\widetilde{y}$, $\widetilde{z}\in(-1,1)$ with the constant $\sqrt{5}\slash 2$. Therefore, we have $c^*\le \sqrt{5}\slash 2$.

To prove that $c^*=\sqrt{5}\slash2$, consider $\Delta=[-1,1]\subset \R^1$ as a part of $\R^n$. Let $\widetilde{u}=-1$ and $\widetilde{v}=1$ be  the endpoints of $\Delta$. Consider the domain $G_1=G_1^{\widetilde{u},\widetilde{v}}$. For all $\widetilde{x}$, $\widetilde{y} \in\Delta$ we have $p_{G_1}(\widetilde{x},\widetilde{y})=p_{\Delta}(\widetilde{x},\widetilde{y})$. Since in  \eqref{cp} the constant $\sqrt{5}\slash 2$ is sharp if we take $\widetilde{x}$, $\widetilde{y}$ and $\widetilde{z}$ from $(-1,1)$, we obtain that it is sharp for $G_1$  and, therefore, for the class of proper subdomains on $\R^n$. The theorem is proved.
\end{proof}

Now, we will investigate the sharpness of the constant $\sqrt{5}\slash 2$, if a proper subdomain $G$ of $\R^n$ is fixed.

\begin{lemma}\label{lem_csharpness}
If a domain $G\subsetneq\R^n$, $n\geq1$, contains some ball $B^n(z_0,r)$ and there are two points $u$, $v\in\partial G$ such that the segment $[u,v]$ is a diameter of $B^n(z_0,r)$, then $c=\sqrt{5}\slash2$ is the best possible constant for which the inequality
$$
p_G(x,y)\le c(p_G(x,z)+p_G(z,y)), \quad x,y,z
\in G,
$$
is valid.
\end{lemma}

\begin{proof}
By Theorem \ref{thm_ppfquasiconstant}, the point pair function is a quasi-metric with the constant $\sqrt{5}\slash2$. The sharpness of this constant follows from the fact that the equality
\begin{equation}\label{52G}
p_G(x,y)=(\sqrt{5}\slash2)(p_G(x,z)+p_G(z,y))
\end{equation}
holds for the points $x=z_0+(u-z_0)\slash3$, $z=z_0$ and $y=z_0+(v-z_0)\slash3$ (see Figure~\ref{fig1a}).
\end{proof}

\begin{figure}[ht]
\scalebox{0.4}{\includegraphics{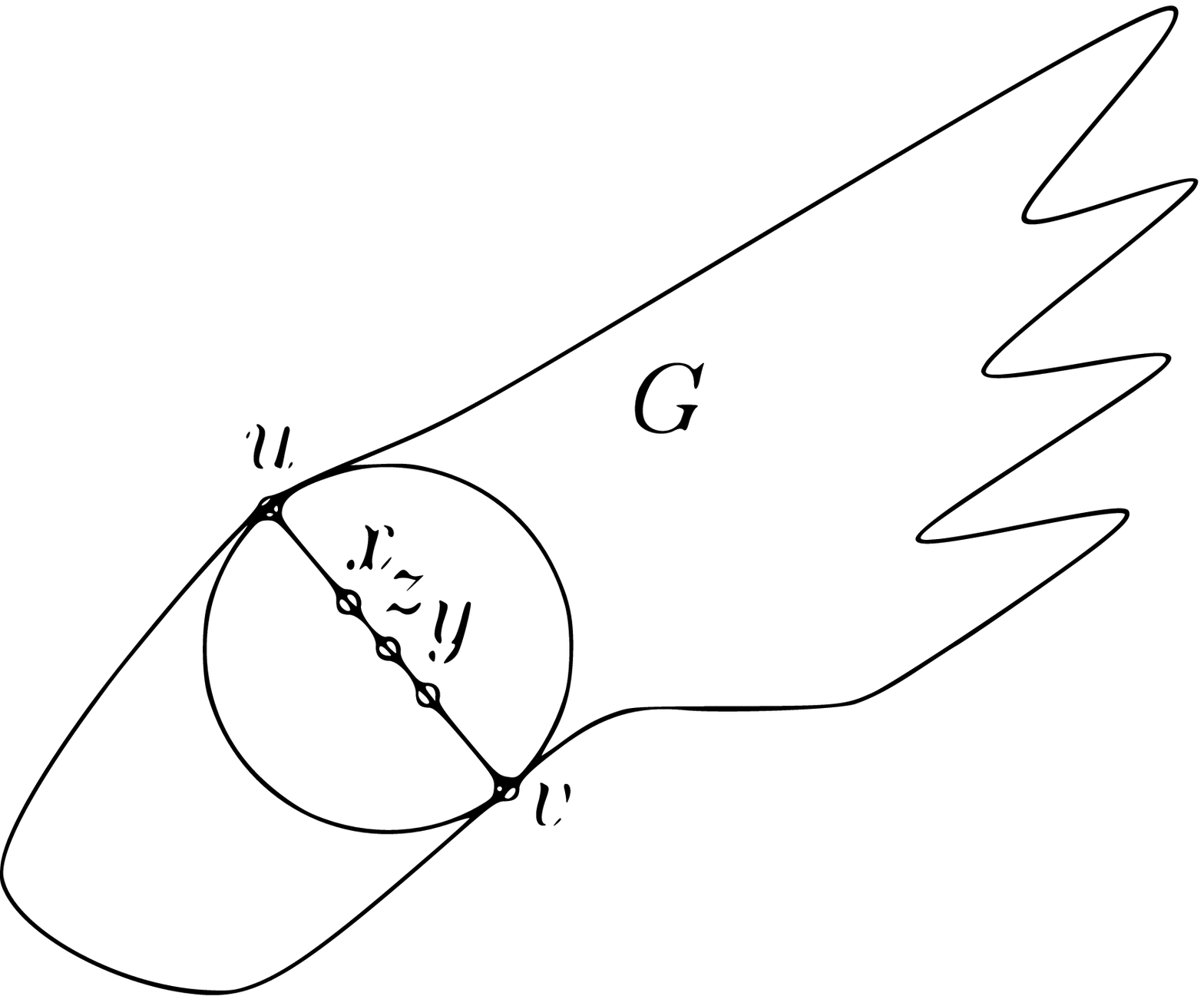}}
       \caption{A domain $G$ and points $x,y,z\in G$ for which the equality \eqref{52G} holds.}
       \label{fig1a}
\end{figure}

It follows from Lemma \ref{lem_csharpness} that the point pair function $p_G$ is a quasi-metric with the best possible constant $\sqrt{5}\slash2$ if the domain $G$ is, for instance, a ball, a hypercube, a hyperrectangle, a multipunctured real space of any dimension $n\geq1$, or a two-dimensional, regular and convex polygon with an even number of vertices.


\section{The generalized point pair function}


In this section, we will consider the generalized version of the point pair function. Namely, note that, by replacing the constant 4 with some $\alpha>0$, we obtain the function
\begin{align}\label{ppfalpha}
p^\alpha_G(x,y)=\frac{|x-y|}{\sqrt{|x-y|^2+\alpha d_G(x)d_G(y)}}\,.  \end{align}
Let us first consider the case where the domain $G$ is the positive real axis.

\begin{theorem}\label{thm_alphappfR}
For a constant $\alpha>0$, the function
\begin{align*}
p^\alpha_{\R^+}(x,y)=\frac{|x-y|}{\sqrt{(x-y)^2+\alpha xy}}\,,
\quad
x,y>0,
\end{align*}
is a metric if and only if $\alpha\leq12$.
\end{theorem}
\begin{proof}
To prove that for every fixed $0<\alpha\leq12$, the function $p^\alpha_{\R^+}(x,y)$ is a metric on the positive real axis, it is sufficient to establish the triangle inequality $p^\alpha_{\R^+}(x,y)\leq p^\alpha_{\R^+}(x,z)+p^\alpha_{\R^+}(z,y)$. Fix first two points $x,y>0$. By symmetry, we can assume that $x\leq y$. Next, we fix $z$ such that the sum $p^\alpha_{\R^+}(x,z)+p^\alpha_{\R^+}(z,y)$ is at minimum. Without loss of generality, we can assume that $0<x\leq z\leq y$ because, for all $z\in(0,x)$,
$$
p^\alpha_{\R^+}(x,2x-z)+p^\alpha_{\R^+}(2x-z,y)<p^\alpha_{\R^+}(x,z)+p^\alpha_{\R^+}(z,y),
$$
and, if $y<z$, then the triangle inequality $p^\alpha_{\R^+}(x,y)\leq p^\alpha_{\R^+}(x,z)+p^\alpha_{\R^+}(z,y)$ holds trivially. Because the function $p^\alpha_{\R^+}$ is invariant under any stretching by any factor $r>0$, we can assume that $z=1$.

Our aim is to prove $p^\alpha_{\R^+}(x,y)\leq p^\alpha_{\R^+}(x,1)+p^\alpha_{\R^+}(1,y)$ or, equivalently,
\begin{align}\label{F}
F(x,z):=\frac{1-x}{\sqrt{(1-x)^2+\alpha x}}\,+\,\frac{y-1}{\sqrt{(y-1)^2+\alpha y}}\,-\,\frac{y-x}{\sqrt{(x-y)^2+\alpha xy}}\,\ge0,
\end{align}
for $0<x\leq1\leq y$.
By denoting $u=1/\sqrt{x}$, $v=\sqrt{y}$, we can  write \eqref{F} as
\begin{align}\label{G}
G(u,v):=\frac{u-u^{-1}}{\sqrt{(u-u^{-1})^2+\alpha}}\,+\,\frac{v-v^{-1}}{\sqrt{(v-v^{-1})^2+\alpha }}\,-\,\frac{uv-(uv)^{-1}}{\sqrt{(uv-(uv)^{-1})^2+\alpha}}\,\ge 0,
\end{align}
$u$, $v\ge 1$.
Now, we will find the critical points of $G(u,v)$. We have
\begin{align*}
\frac{\partial G(u,v)}{\partial u}&=\,\alpha\, \frac{1+u^{-2}}{\sqrt{((u-u^{-1})^2+\alpha)^3}}-\alpha v\,\frac{1+(uv)^{-2}}{\sqrt{((uv-(uv)^{-1})^2+\alpha)^3}},\\
\frac{\partial G(u,v)}{\partial v}&=\,\alpha\, \frac{1+v^{-2}}{\sqrt{((v-v^{-1})^2+\alpha)^3}}-\alpha u\,\frac{1+(uv)^{-2}}{\sqrt{((uv-(uv)^{-1})^2+\alpha)^3}}.
\end{align*}
If
$$
\frac{\partial G(u,v)}{\partial u}\,=\frac{\partial G(u,v)}{\partial v}\,=0,
$$
then it is easy to show that
$$
\frac{u+u^{-1}}{\sqrt{((u-u^{-1})^2+\alpha)^3}}=\frac{v+v^{-1}}{\sqrt{((v-v^{-1})^2+\alpha)^3}}\,,
$$
consequently,
\begin{equation}\label{uuminusone}
\frac{u+u^{-1}}{\sqrt{((u+u^{-1})^2+\alpha-4)^3}}=\frac{v+v^{-1}}{\sqrt{((v+v^{-1})^2+\alpha-4)^3}}\,.
\end{equation}
The function 
$t\mapsto t(t^2+\alpha-4)^{-3\slash2}$ is monotonic on $[2,\infty)$. Since $u+u^{-1}\ge 2$ and $v+v^{-1}\ge 2$ for $u$, $v\geq1$, from \eqref{uuminusone} we deduce that $u+u^{-1}=v+v^{-1}$. From the monotonicity of the function $u+u^{-1}$ on $[1,\infty)$, it follows that $u=v$. Consequently, all the critical points of $G(u,v)$ are on the line $u=v$.

If $u\to u_0$ or $v\to v_0$ and either $u_0$ or $v_0$ equals $1$, then  $G(u,v)$ tends to a non-negative value. Similarly, this condition holds if $u_0$ or $v_0$ equals $+\infty$. Therefore, to prove that the inequality \eqref{G} holds we only need to show that $G(u,u)\ge 0$, $u\ge 1$ or, equivalently,
$$
2\,\frac{u-u^{-1}}{\sqrt{(u-u^{-1})^2+\alpha}}\,-\,\frac{u^2-u^{-2}}{\sqrt{(u^2-u^{-2})^2+\alpha}}\,\ge0.
$$

This inequality can be written as
$$
2\sqrt{(u^2-u^{-2})^2+\alpha}\ge(u+u^{-1})\sqrt{(u-u^{-1})^2+\alpha}
$$
or
\begin{equation}\label{44}
4(u^4+u^{-4}+\alpha-2)\ge(u^2+u^{-2}+2)(u^2+u^{-2}+\alpha-2).
\end{equation}
By denoting $t=u^2+u^{-2}$, we will have $t\ge 2$, and the inequality \eqref{44} takes the form
$$
4(t^2+\alpha-4)\ge(t+2)(t+\alpha-2)
$$
or, equivalently,
\begin{equation}\label{3t2}
3 t^2-\alpha t+2\alpha-12=(t-2)(3t-(\alpha-6))\ge 0, \quad t\ge 2.
\end{equation}
The inequality \eqref{3t2} is valid for $\alpha\ge12$ and $t\ge 2$ because for such $\alpha$ and $t$ we have
$$
3t-(\alpha-6)\ge 6-(\alpha-6)=12-\alpha\ge 0.
$$
Consequently, the inequality \eqref{F} holds and, for $0<\alpha\leq12$, the function $p^\alpha_{\R^+}$ is a metric. It also follows that the constant $12$ here is sharp because, for $\alpha>12$, we have $$
3 t^2-\alpha t+2\alpha-12<0, \quad  2<t<(\alpha-6)/3,
$$
and, therefore, the inequality \eqref{3t2} is not valid at every point of $[2,+\infty)$.
\end{proof}

In Theorem~\ref{pametr}, we prove a result about the function $p_G^\alpha$ similar to Theorem \ref{thm_alphappfR} but for the case where $G=\R^n\backslash\{0\}$. See Figure \ref{fig2} for the disks of the function $p_G^\alpha$ in $\R^2\backslash\{0\}$. However, in order to prove Theorem \ref{pametr}, we need to first consider the following lemma.

\begin{lemma}\label{style41}
If $A,B,C>0$ and $a,b,c\geq0$ are chosen so that the inequalities
\begin{equation}\label{ABC2}
\frac{A}{\sqrt{1+A^2}}\,\le\, \frac{B}{\sqrt{1+B^2}}\,+\,\frac{C}{\sqrt{1+C^2}}\,
\end{equation}
and $a\le b+c$ hold, then
$$
{\sqrt\frac{A^2+a^2}{1+A^2+a^2}}\,\le\, \sqrt {\frac{B^2+b^2}{1+B^2+b^2}}\,+\,\sqrt\frac{C^2+c^2}{1+C^2+c^2}\,.
$$
\end{lemma}

\begin{proof}
Since the functions $t\mapsto t/\sqrt{1+t^2}$ and $t\mapsto\sqrt{t/(1+t)}$ are increasing on $[0,\infty)$, we can assume that the equality takes place in \eqref{ABC2}.

Consider the function
$$
F(x,y):=\, \sqrt {\frac{B^2+x^2}{1+B^2+x^2}}\,+\,\sqrt\frac{C^2+y^2}{1+C^2+y^2}\,-\,{\sqrt\frac{A^2+(x+y)^2}{1+A^2+(x+y)^2}}\,.
$$
We need to prove that $F(x,y)\ge 0$, $x$, $y\ge0$. We have $F(0,0)=0$. Assume that for some $x$, $y\ge0$, not equal to zero at the same time,
$F(x,y)<0$. Consider now the function $g(t)=F(tx,ty)$, $t\in[0,1]$. It is continuous and $g(0)=0$, $g(1)<0$. We will show that for small positive $t$, the inequality $g(t)>0$ holds.
Actually,
$$
g'(t)/t=\frac{x^2}{\sqrt{(B^2+t^2x^2)(1+B^2+t^2x^2)^3}}\,+\frac{y^2}{\sqrt{(C^2+t^2y^2)(1+C^2+t^2y^2)^3}}$$
$$
-\,\frac{(x+y)^2}{\sqrt{(A^2+t^2(x+y)^2)(1+A^2+t^2(x+y)^2)^3}}\,$$$$\to\frac{x^2}{B\sqrt{(1+B^2)^3}}\,+\frac{y^2}{C\sqrt{(1+C^2)^3}}\,-\frac{(x+y)^2}{A\sqrt{(1+A^2)^3}}\,,\quad\text{as}\quad  t\to 0.
$$
Now, we will show that
\begin{equation}\label{xByC}
\frac{x^2}{B\sqrt{(1+B^2)^3}}\,+\,\frac{y^2}{C\sqrt{(1+C^2)^3}}\,-\,\frac{(x+y)^2}{A\sqrt{(1+A^2)^3}}\,>0.
\end{equation}
Denote
$$
a_1=\frac{A}{\sqrt{1+A^2}}\,, \quad b_1=\frac{B}{\sqrt{1+B^2}}\,, \quad c_1=\frac{C}{\sqrt{1+C^2}}\,.
$$
Then $a_1=b_1+c_1<1$ and
$$
A=\frac{a_1}{\sqrt{1-a_1^2}}\,, \quad B=\frac{b_1}{\sqrt{1-b_1^2}}\,, \quad C=\frac{c_1}{\sqrt{1-c_1^2}}\,.
$$
$$
\sqrt{1+A^2}=\frac{1}{\sqrt{1-a_1^2}}\,, \quad \sqrt{1+B^2}=\frac{1}{\sqrt{1-b_1^2}}\,, \quad \sqrt{1+C^2}=\frac{1}{\sqrt{1-c_1^2}}\,.
$$
In this notation, the inequality \eqref{xByC}
can be written in the form
$$
\frac{x^2(1-b_1^2)^2}{b_1}\,+\,\frac{y^2(1-c_1^2)^2}{c_1}\,-\,\frac{(x+y)^2(1-a_1^2)^2}{a_1}\,>\,0.
$$
It is easy to prove that
$$
\frac{(x+y)^2}{a_1}\,=\,\frac{(x+y)^2}{b_1+c_1}\,\le \,\frac{x^2}{b_1}\,+\,\frac{y^2}{c_1}\,,
$$
and it is therefore sufficient to show that
$$
{{(1-a_1^2)^2}}< {(1-b_1^2)^2}\,, \quad {{(1-a_1^2)^2}}< {{(1-c_1^2)^2}},
$$
which follows from the fact that $a_1>b_1$ and $a_1>c_1$.

Since for small positive $t$, we have $g(t)>0$, and $g(1)<0$, we conclude that there exists $t_0\in (0,1)$ and $\varepsilon>0$ such that $g(t_0)=0$ and $g(t)<0$ on $(t_0,t_0+\varepsilon)$. Therefore, $g'(t_0)\le 0$. Denote $x_0=t_0x$,  $y_0=t_0y$. Then
$$
{\sqrt\frac{A^2+(x_0+y_0)^2}{1+A^2+(x_0+y_0)^2}}\,=\,\sqrt {\frac{B^2+x_0^2}{1+B^2+x_0^2}}\,+\,\sqrt\frac{C^2+y_0^2}{1+C^2+y_0^2}\,.
$$
By denoting
$$
{A_1}=\sqrt{A^2+(x_0+y_0)^2}\,, \quad {B_1}=\sqrt{B^2+x_0^2}\,,\quad {C_1}=\sqrt{C^2+y_0^2}\,,
$$
we have
$$
\frac{{A_1}}{\sqrt{1+A_1^2}}\,=\frac{{B_1}}{\sqrt{1+B_1^2}}\,+\frac{{C_1}}{\sqrt{1+C_1^2}}\,.
$$
$$
g'(t_0)/t_0
=\frac{x^2}{{B_1}\sqrt{(1+B_1^2)^3}}\,+\frac{y^2}{{C_1}\sqrt{(1+C_1^2)^3}}\,-\frac{(x+y)^2}{{A_1}\sqrt{(1+A_1^2)^3}}\,.
$$
Reasoning as above but by replacing $A$, $B$ and $C$ with ${A_1}$, ${B_1}$ and  ${C_1}$, we show that $g'(t_0)>0$. The contradiction proves the theorem.
\end{proof}

\begin{figure}[ht]
    \scalebox{0.3}{\includegraphics{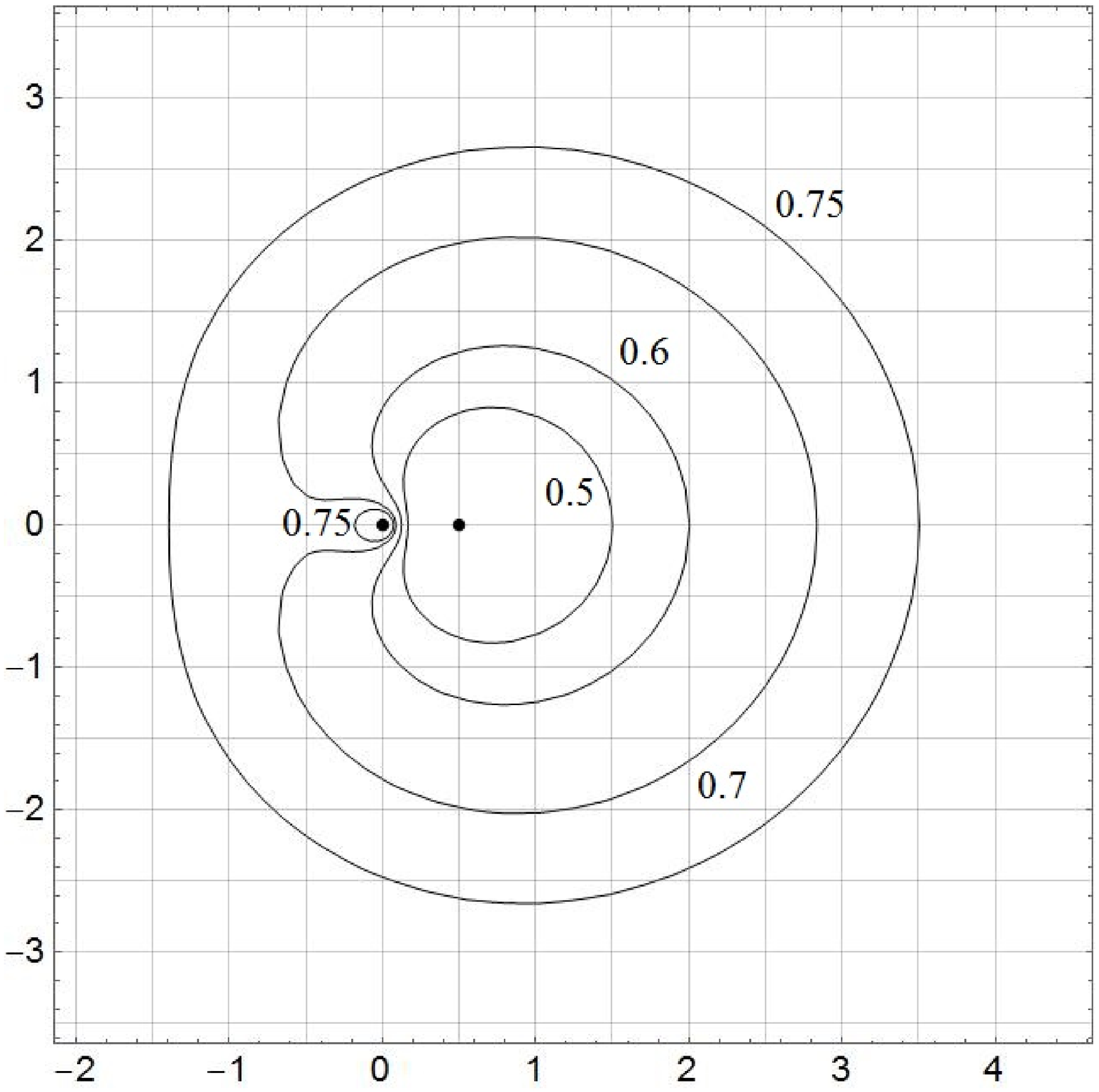}}
           \caption{Disks $\{x\in G\text{ }|\text{ }p^\alpha_G(x,0.5)<r\}$ with center $0.5$ and radii $r=0.5,0.6,0.7,0.75$ for the metric $p^\alpha_G$, $\alpha=3.5$, in the domain $G=\R^2\backslash\{0\}$.}
       \label{fig2}
\end{figure}

\begin{theorem}\label{pametr}
For a constant $\alpha>0$, the function
\begin{align*}
p^\alpha_{\R^n\backslash\{0\}}(x,y)=\frac{|x-y|}{\sqrt{|x-y|^2+\alpha|x|\,|y|}},
\quad
x,y\in\R^n\backslash\{0\},\, n\geq2,
\end{align*}
is a metric if and only if $\alpha\leq12$.
\end{theorem}
\begin{proof}
We will only outline the proof because it is similar to that of Theorem~\ref{thm_pInR0} but, instead of \eqref{ABCdef}, we use the following values:
\begin{equation}\label{ABCdef1}
A=\frac{|R-r|}{\sqrt{\alpha Rr}}\,, \quad B=\frac{|R-1|}{\sqrt{\alpha R}}\,, \quad
C=\frac{|1-r|}{\sqrt{\alpha r}},
\end{equation}
By Theorem~\ref{thm_alphappfR}, such values satisfy \eqref{ABC2}. We also note that the parameters $a$, $b$ and $c$, considered in both the first and second parts of the proof of Theorem~\ref{thm_pInR0}, satisfy the inequality $0\le a
\le b+c$. Therefore, we can apply Lemma~\ref{style41} instead of Lemma~\ref{lem_trigABC}, to prove \eqref{ABC} and \eqref{ABC1} and  establish the triangle inequality.
\end{proof}

Now, we will study a generalization of the point pair function in the upper half-space $\uhp^n$.

\begin{theorem}\label{thm_mppfInH}
For a constant $\alpha>0$, the function
\begin{align}
p^\alpha_{\uhp^n}(x,y)=\frac{|x-y|}{\sqrt{|x-y|^2+\alpha x_ny_n}}\,,
\quad
x=(x_1,...,x_n),y=(y_1,...,y_n)\in\uhp^n,\, n\geq2,
\end{align}
is a metric if and only if $\alpha\leq12$.
\end{theorem}
\begin{proof}
It is evident that this function fulfills the first two properties of a metric if $\alpha>0$ and, therefore, we only need to investigate the fulfillment of the triangle inequality.
If we consider some points $x=(0,0,\ldots,0,x_n)$ and $y=(0,0,\ldots,0,y_n)$ 
with $x_n$, $y_n>0$, then  $p^\alpha_{\uhp^n}(x,y)=p^\alpha_{\R^+}(x_n,y_n)$ and, from Theorem~\ref{thm_alphappfR}, it follows that the function $p^\alpha_{\uhp^n}$ is not a metric for $\alpha>12$,
since $p^\alpha_{\R^+}$ is not a metric for such $\alpha$.

Now, we will consider the case $\alpha\in (0,12]$ and prove that, in this case, $p^\alpha_{\uhp^n}$ is a metric. Let us fix some distinct points $x$, $y \in \uhp^n$. Consider a two-dimensional plane $\Pi$ in $\R^n$, which is orthogonal to the hyperplane $\{x_n=0\}$ and contains the points $x$ and $y$. If $x$ and $y$ do not lie on the same line, orthogonal to  $\{x_n=0\}$, then $\Pi$ is defined in a unique way; in the opposite case, we fix any of the possible planes. Now, we take any point $z\in \uhp^n$ and find its orthogonal projection $\widetilde{z}$ to $\Pi$. Since $|\widetilde{z}-x|\le|z-x|$, $|\widetilde{z}-y|\le|z-y|$ and $\widetilde{z}_n=z_n$, we obtain $p^\alpha_{\uhp^n}(x,z)+p^\alpha_{\uhp^n}(z,y)\ge p^\alpha_{\uhp^n}(x,\widetilde{z})+p^\alpha_{\uhp^n}(\widetilde{z},y)$. Therefore, we only need to prove that $p^\alpha_{\uhp^n}(x,\widetilde{z})+p^\alpha_{\uhp^n}(\widetilde{z},y)\ge p^\alpha_{\uhp^n}(x,y)$ for every three points $x$, $y$ and $\widetilde{z}$ lying in a two-dimensional plane $\Pi$.

Thus, without loss of generality we can assume that $n=2$ and the points $x$, $y$ and $z \in \uhp^2$ are complex numbers.

Consider the two following cases. \smallskip

1) If $\Re x=\Re y$, then we denote by  $L$  the line, orthogonal to the real axis and containing $x$ and $y$, and replace $z$ with its orthogonal projection $z'$ to $L$. Reasoning as above, we have $p^\alpha_{\uhp^2}(x,z)+p^\alpha_{\uhp^2}(z,y)\ge p^\alpha_{\uhp^2}(x,{z'})+p^\alpha_{\uhp^2}({z'},y)$, and, therefore, we can reduce the problem to the one-dimensional case. From Theorem~\ref{thm_alphappfR}, it follows that $p^\alpha_{\uhp^2}(x,{z'})+p^\alpha_{\uhp^2}({z'},y)=p^\alpha_{\R^+}(\Im x,{\Im z'})+p^\alpha_{\R^+}({\Im z'},\Im y)\ge p^\alpha_{\R^+}({\Im x},\Im y)=p^\alpha_{\uhp^2}(x,{y})$ and the triangle inequality is valid.\smallskip

2) If $\Re x\neq\Re y$, then we consider the circle $C$ containing the points $x$ and $y$ and orthogonal to the real axis. Let $C^+$ be the upper half of $C$. There exists a M\"obius transformation $T$ that maps $C^+$ onto the positive part of the imaginary axis. Consider the points $x_1=T(x)$, $y_1=T(y)$ and $z_1=T(z)$.
For every $u$, $v\in \uhp^2$, $u\neq v$, we have $p^\alpha_{\uhp^2}(u,v)=(1+\alpha t)^{-1/2}$ where $t=(\Im u\Im v)\slash(|u-v|^2)$. According to the well-known property of M\"obius automorphisms of $\uhp^2$,
$$
\frac{\Im u\Im v}{|u-v|^2}=\frac{\Im T(u)\Im T(v)}{|T(u)-T(v)|^2}\,,
$$
and we can conclude that $T$ preserves the value $p^\alpha_{\uhp^2}(u,v)$, i.e. $p^\alpha_{\uhp^2}(T(u),T(v))=p^\alpha_{\uhp^2}(u,v)$. Making use of this fact, we can replace $x$, $y$ and $z$ with $x_1$, $y_1$ and $z_1$; but for such points $\Re x_1=\Re y_1$ and the triangle inequality, therefore, follows from Case 1).
\end{proof}

It is interesting to study whether the point pair function defined as in \eqref{alpha4} becomes a metric, if we replace the constant $4$ with a smaller positive constant. The answer is negative, as proven below.

\begin{theorem}\label{thm_alphappfnotmetricInB}
The function
\begin{align*}
p^\alpha_{\B^n}(x,y)=\frac{|x-y|}{\sqrt{|x-y|^2+\alpha(1-|x|)(1-|y|)}}\,,
\quad
x,y\in\B^n,\, n\geq1,
\end{align*}
is not a metric for any constant $\alpha>0$.
\end{theorem}
\begin{proof}
Fix points $x=ke_1$, $y=-ke_1$, and $z=0$, where $0<k<1$ and $e_1$ is the first unit vector. Then we have
\begin{multline*}
p^\alpha_{\B^n}(x,y)> p^\alpha_{\B^n}(x,z)+p^\alpha_{\B^n}(z,y)
\quad\Leftrightarrow\quad
\frac{p^\alpha_{\B^n}(x,y)}{p^\alpha_{\B^n}(x,z)+p^\alpha_{\B^n}(z,y)}=
\sqrt{\frac{k^2+\alpha(1-k)}{4k^2+\alpha(1-k)^2}}>1\\
\Leftrightarrow\quad
k^2+\alpha(1-k)>4k^2+\alpha(1-k)^2
\quad\Leftrightarrow\quad
-3k+\alpha(1-k)>0
\quad\Leftrightarrow\quad
k<\alpha/(3+\alpha).
\end{multline*}
Consequently, if we put $x=ke_1$, $y=-ke_1$, $k=\alpha\slash(4+\alpha)$ and $z=0$, then the triangle inequality does not hold.
\end{proof}


\section{Open questions}

On the base of numerical tests, we propose the following conjectures.

\begin{conjecture}\label{conj_ppfRminusB}
For a constant $\alpha>0$, the function
\begin{align}
p^\alpha_{\R^n\backslash\overline{\B}^n}(x,y)=\frac{|x-y|}{\sqrt{|x-y|^2+\alpha(|x|-1)(|y|-1)}}\,,
\quad
x,y\in\R^n\backslash\overline{\B}^n,\,n\geq2,
\end{align}
is a metric if and only if $\alpha\leq12$.
\end{conjecture}

\begin{remark}
Since the function $x\mapsto x/|x|^2$ maps $\R^n\backslash\overline{\B}{}^n$, $n\geq2$, onto the domain $G=\B^n\setminus\{0\}$ and vice versa, it follows that if Conjecture \ref{conj_ppfRminusB} holds, then the quasimetric
\begin{align}\label{ppfInversionDisk}
\psi^\alpha_G(x,y)=
p^\alpha_{\R^n\backslash\overline{\B}^n}\left(\frac{x}{|x|^2},\frac{y}{|y|^2}\right)=
\frac{|x-y|}{\sqrt{|x-y|^2+\alpha|x|\,|y|(1-|x|)(1-|y|)}},
\quad x,y\in G,
\end{align}
is a metric on $G$ if and only if $\alpha\in(0,12]$. See Figure \ref{fig3} for the  `disks' of this function $\psi^\alpha_G$.
\end{remark}


\begin{figure}[ht]
    \scalebox{0.7}{\includegraphics{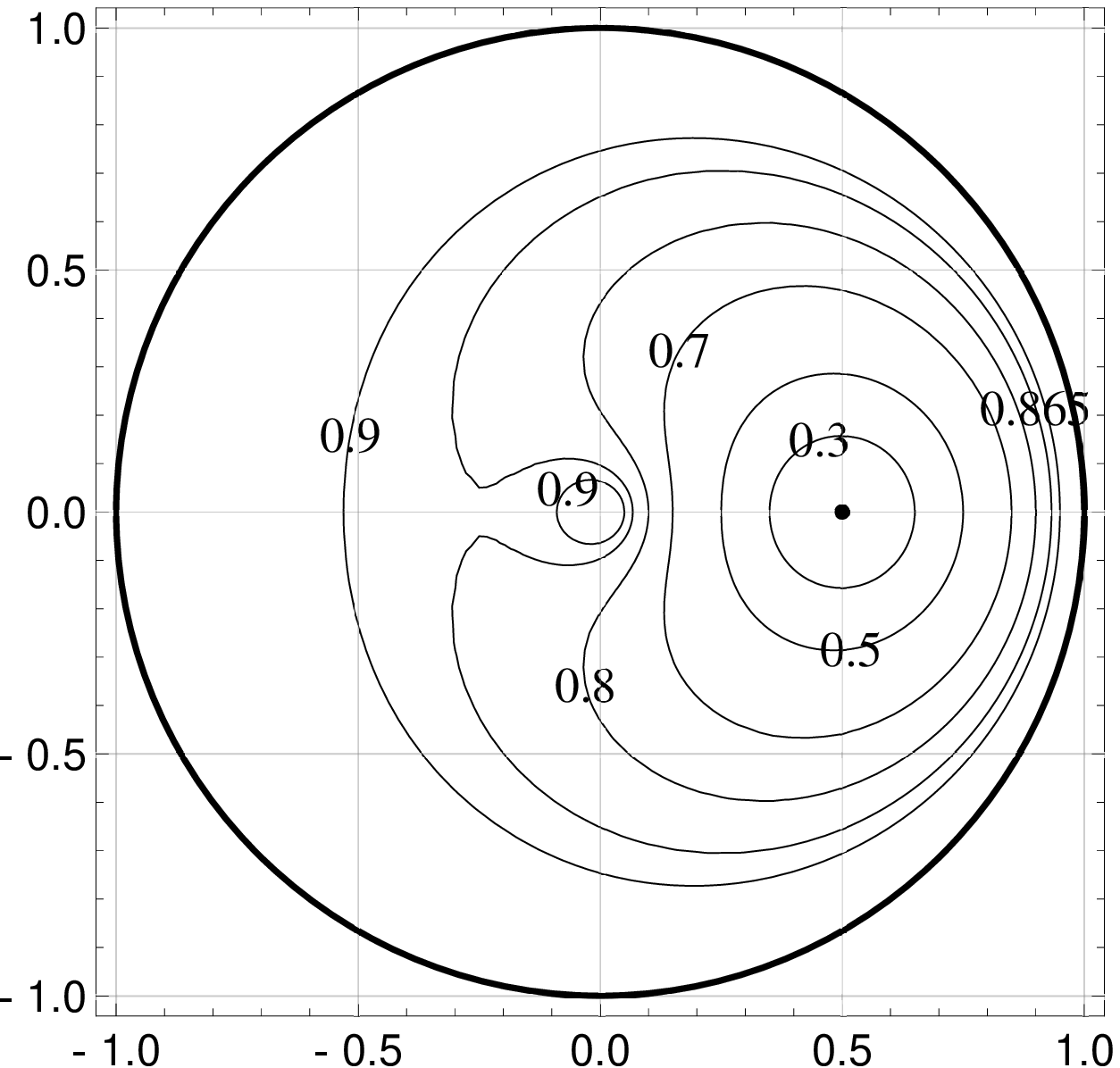}}
           \caption{Disks $\{x\in G\text{ }|\text{ }\psi^\alpha_G(x,0.5)<r\}$ with center $0.5$ and radii $r=0.3,0.5,0.7,0.8,0.865,0.9$ for the function $\psi^\alpha_G$, $\alpha=4$, and $G=\B^2\setminus\{0\}$, in the unit disk $\B^2$.}
       \label{fig3}
\end{figure}

\begin{conjecture}
The point pair function $p_G$ is a metric on the domain $G=\R^3\backslash Z$, where $Z$ is the $z$-axis.
\end{conjecture}

\begin{remark}
From the proof of Theorem \ref{thm_alphappfnotmetricInB} it follows that, for $\alpha>0$, the generalized version $p^\alpha_{\B^n}$ of the point pair function can only be a quasi-metric in the unit ball with a constant $c(\alpha)$ that has the following lower bound:
\begin{align}\label{ine_c_alphakformula}
c(\alpha)\geq\sup_{0<k<1}\sqrt{\frac{k^2+\alpha(1-k)}{4k^2+\alpha(1-k)^2}}.
\end{align}
By differentiation, we have
\begin{align*}
\frac{\partial}{\partial k}\left(\frac{k^2+\alpha(1-k)}{4k^2+\alpha(1-k)^2}\right)
=\frac{\alpha((\alpha+2)k^2-2(\alpha+3)k+\alpha)}{(4k^2+\alpha(1-k)^2)^2}=0\\
\Leftrightarrow\quad
(\alpha+2)k^2-2(\alpha+3)k+\alpha=0
\quad\Leftrightarrow\quad
k=\frac{\alpha+3\pm\sqrt{4\alpha+9}}{\alpha+2}\,.
\end{align*}
It can be shown that the square-root expression on the right hand side of the inequality \eqref{ine_c_alphakformula} obtains its maximum with respect to $k\in(0,1)$ at the point
\begin{align*}
k=\frac{\alpha+3-\sqrt{4\alpha+9}}{\alpha+2}\in(0,1).
\end{align*}
Consequently, the inequality \eqref{ine_c_alphakformula} can be simplified to $c(\alpha)\ge c_*(\alpha)$ where
\begin{align*}
c_*(\alpha)=\sqrt{\frac{(\alpha+3-\sqrt{4\alpha+9})^2+\alpha(\alpha+2)(\sqrt{4\alpha+9}-1)}{4(\alpha+3-\sqrt{4\alpha+9})^2+\alpha(1-\sqrt{4\alpha+9})^2}}\,.
\end{align*}
\end{remark}

\begin{conjecture}
For $\alpha>0$, the function
\begin{align*}
p^\alpha_{\B^n}(x,y)=\frac{|x-y|}{\sqrt{|x-y|^2+\alpha(1-|x|)(1-|y|)}}\,,
\quad
x,y\in\B^n,\, n\geq1,
\end{align*}
is a quasi-metric with the sharp constant $c_*(\alpha)$
or, equivalently, $c_*(\alpha)$ defined as above is the best of constants $c(\alpha)$ depending only on the value of $\alpha$ such that the inequality
\begin{align*}
p^\alpha_{\B^n}(x,y)\leq c(\alpha)(p^\alpha_{\B^n}(x,z)+p^\alpha_{\B^n}(z,y))
\end{align*}
holds for all $x$, $y$, $z\in\B^n$ and $n\geq1$.
\end{conjecture}


{\small
\noindent\textbf{Declarations:}
\textbf{Availability of data and material} Not applicable, no new data was generated.
\textbf{Competing interests} On behalf of all the authors, the corresponding author states that there are no competing interest.
\textbf{Funding} The work of the second author is performed under the development program of Volga Region Mathematical Center (agreement no.~075-02-2022-882). The research of the third author was funded by the University of Turku Graduate School UTUGS.
\textbf{Authors' contributions} DD contributed new ideas and checked the results. SN organized this research and contributed several theorems. OR suggested research ideas and contributed several results. MV did several experiments and suggested problems.
\textbf{Acknowledgments} The authors are grateful to the referees for their work. 
}

\def\cprime{$'$} \def\cprime{$'$} \def\cprime{$'$}
\providecommand{\bysame}{\leavevmode\hbox to3em{\hrulefill}\thinspace}
\providecommand{\MR}{\relax\ifhmode\unskip\space\fi MR }
\providecommand{\MRhref}[2]{%
  \href{http://www.ams.org/mathscinet-getitem?mr=#1}{#2}
}
\providecommand{\href}[2]{#2}

\end{document}